\documentclass[a4paper,11pt]{amsart}
\usepackage{amssymb}
\usepackage[alwaysadjust]{paralist}
\usepackage{cleveref}
\numberwithin{equation}{section}

\newtheorem{cor}[equation]{Corollary}
\newtheorem{lem}[equation]{Lemma}

\newtheorem{thm}[equation]{Theorem}
\theoremstyle{definition}

\def\R{\mathbb R}

\newcommand{\id}{\operatorname{id}}

\begin{document}


\title{Small eigenvalues of surfaces}
\author{Werner Ballmann}
\address
{WB: Hausdorff Center for Mathematics,
Endenicher Allee 60, 53115 Bonn,
and Max Planck Institute for Mathematics,
Vivatsgasse 7, 53111 Bonn.}
\email{hwbllmnn\@@mpim-bonn.mpg.de}
\author{Henrik Matthiesen}
\address
{HM: Mathematisches Institut, Endenicher Allee 60, 53115 Bonn.}
\email{henrik\_matthiesen\@@gmx.de}
\author{Sugata Mondal}
\address
{SM: Max Planck Institute for Mathematics, Vivatsgasse 7, 53111 Bonn.}
\email{sugata.mondal\@@mpim-bonn.mpg.de}

\thanks{\emph{Acknowledgments.}
We would like to thank the Max Planck Institute for Mathematics in Bonn
for its support and hospitality.}

\date{\today}


\begin{abstract}
We show that the Laplacian of a Riemannian metric on a closed surface $S$
with Euler characteristic $\chi(S)<0$ has at most $-\chi(S)$ small eigenvalues.
\end{abstract}

\maketitle

\section{Introduction}
\label{intro}

Relations between the spectrum of the Laplacian
and the geometry and topology of the underlying Riemannian manifold
are a fascinating topic at the crossroads of a number of mathematical fields.
We are concerned with the case of closed Riemannian surfaces $S$.
Then the spectrum of the Laplacian $\Delta$ is discrete
and consists of eigenvalues with finite multiplicity.
We enumerate these in increasing order,
\[ 0 = \lambda_0 < \lambda_1 \le \lambda_2 \le \dots \]
where an eigenvalue is counted as often as its multiplicity requires.

In \cite{Bu1,Bu2,S1,S2},
Peter Buser and Paul Schmutz studied the Laplace operator
of hyperbolic metrics, that is, Riemannian metrics of constant Gauss curvature $-1$,
on closed orientable surfaces $S=S_g$ of genus $g\ge2$.
Based on their work, they conjectured
that the Laplace operator of a hyperbolic metric on $S_g$ has at most $2g-2$ small eigenvalues.
These are eigenvalues below $1/4$,
the bottom of the $L^2$-spectrum of the Laplacian on the hyperbolic plane.

In \cite{OR},
Jean-Pierre Otal and Eulalio Rosas proved a general version of this conjecture,
namely that $\lambda_{2g-2}>\lambda_0(\tilde S)$
for any real analytic Riemannian metric of negative curvature on $S_g$,
where $\lambda_0(\tilde S)$ is the bottom of the $L^2$-spectrum of the Laplacian
on the universal covering surface $\tilde S$ of $S_g$,
endowed with the pull back of the Riemannian metric of $S_g$.
In his thesis \cite{Mo} (see also \cite{Mo1}), the third named author showed that,
for any hyperbolic metric on $S_g$,
there is a constant $\epsilon$, which only depends on the systole of the metric,
such that $\lambda_{2g-2}\ge1/4+\epsilon$.
In his Bachelor thesis \cite{Ma}, the second named author showed
that the assumption of negative curvature in the result of Otal and Rosas can be omitted.
Since smooth Riemannian metrics can be approximated by real analytic metrics,
the latter result implies that
$\lambda_{2g-2}\ge\lambda_0(\tilde S)$ for any smooth Riemannian metric on $S_g$.
Our main result, \cref{theo} below, improves this weak to a strict inequality.

For a domain $\Omega$ in a Riemannian surface,
the \emph{bottom of the $L^2$-spectrum} of the Laplacian on $\Omega$ is given by
\begin{equation}\label{bottom}
  {\lambda_0}(\Omega) = \inf\{\mathcal{R}(\phi)\mid\phi\in C^\infty_{cc}(\Omega), \phi\ne0\},
\end{equation}
where $\mathcal R(\phi)=\int\|\nabla\phi\|^2/\int\phi^2$
denotes the \emph{Rayleigh quotient} of $\phi$.
If $\Omega$ is compact with piecewise smooth boundary,
then $\lambda_0(\Omega)$ is {the first Dirichlet eigenvalue} of $\Omega$,
that is, $\lambda_0(\Omega)$ is the smallest $\lambda\in\R$ such that the problem
\begin{equation}
  \Delta\phi = {\lambda\phi} \quad\text{on $\mathring\Omega$}, \quad
  \phi = 0 \quad\text{on $\partial \Omega$},
\end{equation}
admits a non-zero solution $\phi$ which is smooth on $\mathring\Omega$
and continuous on $\bar{\Omega}$.
From this characterization and \eqref{bottom} it is evident that,
for any two compact domains $\Omega_1$ and $\Omega_2$ with piecewise smooth boundary,
\begin{equation}\label{dineq}
{\lambda_0}(\Omega_1) > {\lambda_0}(\Omega_2)
\;\;\text{whenever}\;\; \Omega_1 \subsetneq \Omega_2.
\end{equation}
Suppose now that $S$ is a closed Riemannian surface and let
\begin{equation}\label{dam}
  \Lambda_D(S) = \inf\lambda_0(\Omega), \:
  \Lambda_A(S) = \inf\lambda_0(\Omega), \;
  \Lambda_C(S) = \inf\lambda_0(\Omega),
\end{equation}
where $\Omega$ runs over all compact embedded discs respectively annuli
respectively cross caps in $S$ with piecewise smooth boundary.
In each case, the infimum may also be taken over all finite graphs $G$ in $S$
such that $S\setminus G$ is an open disc, annulus, or cross cap.
Furthermore, we set
\begin{equation}\label{geomc}
  \Lambda(S) = \min\{\Lambda_D(S),\Lambda_A(S),\Lambda_C(S)\}.
\end{equation}
Our main result is the following

\begin{thm}\label{theo}
For any closed surface $S$  with Euler characteristic $\chi(S)<0$, we have
\begin{equation*}
  \lambda_{-\chi(S)}>\Lambda(S).
\end{equation*}
\end{thm}

Observe that any embedded disc, annulus or cross cap with piecewise smooth boundary in $S$
can be lifted isometrically to $\tilde S$ or a cyclic quotient of $\tilde{S}$.
Hence, by Theorem 1 in \cite{Br2} and \eqref{bottom},
we have $\Lambda(S)\ge\lambda_0(\tilde{S})$.
In view of \eqref{dineq} we suspect, however, that always $\Lambda(S)> \lambda_0(\tilde{S})$.
Indeed, for orientable closed surfaces with hyperbolic metrics, we have
\begin{equation}
  \Lambda(S) \ge 1/4 + \delta(S)
\end{equation}
by Theorem 1.1 in \cite{Mo} (or, respectively, Theorem 2.1.4 in \cite{Mo1}),
where
\begin{equation}
\delta(S) = \min \{\frac{\pi}{|S|},\frac{s(S)^2}{|S|^2} \} > 0
\end{equation}
with $s(S)$ and $|S|$ denoting the systole and the area of $S$, respectively.

We emphasize that our proof of \cref{theo} gives the strict inequality
$\lambda_{-\chi(S)}>\lambda_0(\tilde S)$ as opposed to the weak inequality,
which would follow from \cite{Ma}
(at least in the orientable case, as explained further up).
The main line of the proof of \cref{theo} follows \cite{OR}.
As in that reference,
our arguments rely mostly on rather elementary topological properties of surfaces.
However, we do not (and cannot) use the regularity theory of real analytic functions.
Instead, we mostly invoke arguments from the elementary calculus of smooth functions.

\newpage
\section{Approximate nodal sets and domains}
In what follows, $S$ is a closed Riemannian surface with negative Euler characteristic.
We denote by $\Delta$ the positive Laplacian of $S$.
For any $\lambda\ge0$,
we let $\mathbb E_\lambda=\{\phi\in L^2(M)\mid\Delta\phi=\lambda\phi\}$
be the $\lambda$-eigenspace of $\Delta$ in $L^2(M)$
(where we allow for $\mathbb E_\lambda=\{0\}$).
We let
\begin{equation}
  \mathbb E = \oplus_{\lambda\le\Lambda(S)}\mathbb E_\lambda
\end{equation}
and $\mathbb{S}$ be the unit sphere in $\mathbb E$ 
with respect to the $L^2$-norm.
The assertion of \cref{theo} is that $\dim\mathbb E\le-\chi(S)$.

Recall that any eigenfunction of $\Delta$ is smooth (elliptic regularity).
Hence each function in $\mathbb{S}$ is smooth.
For any $\phi \in \mathbb{S}$,
\begin{equation}\label{defnos}
  Z_\phi := \{x\in S\mid \phi(x)=0\}
\end{equation}
is called the \emph{nodal set} of $\phi$.
The connected components of the complement $S\setminus Z_\phi$
are called \emph{nodal domains} of $\phi$.

\begin{lem}\label{lemnae}
For almost any $x\in Z_\phi$, we have $\nabla\phi(x)=0$.
\end{lem}

\begin{proof}
The set of points of density of $Z_\phi$ has full measure in $Z_\phi$,
and, clearly, $\nabla\phi(x)=0$ at any such point $x$.
\end{proof}

We say that $\epsilon>0$ is \emph{regular} or, more precisely, \emph{$\phi$-regular},
if $\epsilon$ and $-\epsilon$ are regular values of $\phi$.
By Sard's theorem, almost any $\epsilon>0$ is regular.

For any $\epsilon>0$, we call
\begin{equation}\label{defzef}
  Z_\phi(\epsilon) := \{x \in S\mid |\phi(x)| \leq \epsilon \}
\end{equation}
the \emph{$\epsilon$-nodal set} of $\phi$.
We are only interested in the case where $\epsilon$ is regular.
Then $Z_\phi(\epsilon)$ is a subsurface of $S$ with smooth boundary,
may be empty or may consist of more than one component,
and the boundary components of $Z_\phi(\epsilon)$ are embedded smooth
circles along which $\phi$ is constant $\pm\epsilon$.

Let $\epsilon>0$ be regular.
Delete from $Z_\phi(\epsilon)$ all the components which are contained
in the interior of an embedded closed disc in $S$
and obtain the \emph{derived $\epsilon$-nodal set} $Z_\phi'(\epsilon)$.
By definition, no component of $Z_\phi'(\epsilon)$
is contained in the interior of an embedded closed disc in $S$.
Since $Z_\phi'(\epsilon)$ is important in our discussion,
we view its definition also from a different angle:
If $D\subseteq S$ is an embedded closed disc,
then the components of $Z_\phi(\epsilon)$ contained in the interior of $D$
are compact and bounded by smooth circles.
Each such circle is the boundary of an embedded closed disc $D'$ in $D$,
by the Schoenfliess theorem.
By definition, the boundary circle $\partial D'$ of any such disc $D'$
is also a boundary circle of a component $C$ of $Z_\phi(\epsilon)$.
There are two possible types for $\partial D'$:
Either $C$ is in the outer part or in the inner part of $\partial D'$ with respect to the interior of $D$.
We say that $D'$ is \emph{essential} if $C$ is in the inner part of $\partial D'$.
In other words, $D'$ is essential if a neighborhood of $\partial D'$ in $D'$
is contained in $Z_\phi(\epsilon)$.
Essential discs in $S$ are either disjoint or one is contained in the other;
they are partially ordered by inclusion.
Therefore each essential disc is contained in a unique maximal essential disc.

For any regular $\epsilon>0$,
$Y_\phi(\epsilon):=S\setminus\mathring{Z}_\phi'(\epsilon)$ is a smooth subsurface of $S$.

\begin{lem}\label{lemyef}
For any regular $\epsilon>0$, we have: \\
\begin{inparaenum}[1)]
\item\label{inco}
$Y_\phi(\epsilon)$ is a smooth and incompressible subsurface of $S$. \\
\item\label{sign}
Each component $C$ of $Y_\phi(\epsilon)$
is the union of some component $C_0$ of $\{\phi\ge\epsilon\}$ or of $\{\phi\le-\epsilon\}$
with a finite number ($\ge0$) of maximal essential discs which are attached to $C'$
along $\partial C'$.
In particular, $\phi|_{\partial C}=+\epsilon$ or $\phi|_{\partial C}=-\epsilon$. \\
\item\label{inty}
The function $\phi_\epsilon$ on $S$, defined by
\begin{equation*}
  \phi_\epsilon(x) = \begin{cases}
  \phi(x)-\epsilon \:&\text{if $\phi(x)\ge\epsilon$}, \\
  \phi(x)+\epsilon &\text{if $\phi(x)\le-\epsilon$}, \\
    \phantom{\phi(x)} 0 &\text{otherwise},
 \end{cases}
\end{equation*}
belongs to $H^1(S)$.
Moreover, $\lim_{\epsilon\rightarrow0}\phi_\epsilon=\phi$ in $H^1(S)$.
\end{inparaenum}
\end{lem}

\begin{proof}
\ref{inco})
Since $Z_\phi'(\epsilon)$ is a union of components of the smooth
subsurface $Z_\phi(\epsilon)$ of $S$, it is a smooth subsurface of $S$.
Hence the complement $Y_\phi(\epsilon)$ of its interior
is also a smooth subsurface of $S$.

It remains to show that there is no loop $c$ in $Y_\phi(\epsilon)$
which is not homotopic to zero in $Y_\phi(\epsilon)$,
but is homotopic to zero in $S$.
We suppose the contrary and assume without loss of generality that $c$ is simple
and contained in the interior of $Y_\phi(\epsilon)$.
Since $c$ is homotopic to zero in $S$, it bounds an embedded closed disc $D$ in $S$,
by \cref{lemdisc}.
Now $D$ is not contained in $Y_\phi(\epsilon)$
since $c$ is not homotopic to zero in $Y_\phi(\epsilon)$.
Hence $D$ contains components of $Z_\phi'(\epsilon)$.
These are in the interior of $D$ since $c$ lies  in the interior of $Y_\phi(\epsilon)$.
But this is in contradiction to the definition of $Z_\phi'(\epsilon)$.

\ref{sign})
For each component of $Z_\phi(\epsilon)$
which is contained in the interior of an embedded closed disc,
choose an essential disc as explained in our discussion of the definition
of $Z_\phi'(\epsilon)$ further up.
Since each essential discs is contained in a unique maximal essential disc,
it follows that $Y_\phi(\epsilon)$ is equal to the (possibly non-disjoint) union
of $S\setminus\mathring{Z}_\phi(\epsilon)=\{|\phi|\ge\epsilon\}$
with maximal essential discs.
Hence each of the components of $Y_\phi(\epsilon)$
consists of some component $C_0$ of $\{\phi\ge\epsilon\}$ or of $\{\phi\le-\epsilon\}$
together with a finite number ($\ge0$) of maximal essential discs
which are attached to $C'$ along $\partial C'$.

\ref{inty})
For all $x\in S$, we have $|\phi_\epsilon(x)|\le|\phi(x)|$.
Hence $\phi_\epsilon$ is in $L^2(M)$.
Moreover, $\phi_\epsilon(x)\rightarrow\phi(x)$ for all $x \in S$,
hence $\lim_{\epsilon\rightarrow0}\phi_\epsilon=\phi$ in $L^2(M)$.
Furthermore, $\phi_\epsilon$ has weak gradient
\begin{equation*}
  \nabla\phi_\epsilon(x) = \begin{cases}
  \nabla\phi(x) \:&\text{if $|\phi(x)|\ge\epsilon$}, \\
  \phantom{\nabla\phi} 0 &\text{otherwise}.
 \end{cases}
\end{equation*}
It follows that $\phi_\epsilon$ is in $H^1(S)$.
Furthermore,
$\lim_{\epsilon\rightarrow0}\nabla\phi_\epsilon=\nabla\phi$ in $H^1(S)$,
by \cref{lemnae}.
\end{proof}

We let $Y_\phi'(\epsilon)$ be the union of the components $C$
of $Y_\phi(\epsilon)$ with Euler characteristic $\chi(C)<0$.
That is, $Y_\phi'(\epsilon)$ is the union of the components of $Y_\phi(\epsilon)$
which are not diffeomorphic to a disc, an annulus, or a cross cap.

\begin{lem}\label{lemeul}
For all sufficiently small regular $\epsilon>0$, we have $\chi(Y_\phi'(\epsilon))<0$.
In other words,
$Y_\phi'(\epsilon)$ is non-empty for all sufficiently small $\epsilon>0$.
\end{lem}

\begin{proof}
{\sc Case 1:}
Assume first that the Rayleigh quotient $R(\phi)<\Lambda(S)$
and choose a $\delta > 0$ such that
\begin{equation*}
  R(\phi) \le \Lambda(S) - 2\delta.
\end{equation*}
By \cref{lemyef}.\ref{inty}, we have, for any sufficiently small regular $\epsilon>0$,
\begin{equation*}
   \frac{\sum_C\int_C|\nabla\phi_\epsilon|^2}{\sum_C\int_C\phi_\epsilon^2}
   \le \frac{\int_S|\nabla\phi|^2 dv}{\int_S\phi^2 dv} + \delta
   = R(\phi) + \delta
   \le \Lambda(S) - \delta,
\end{equation*}
where the sums run over the components $C$ of $Y_\phi(\epsilon)$.
We conclude that there is a component $C$ such that
\begin{equation*}
   R(\phi_\epsilon|_C)
   = \frac{\int_C|\nabla\phi_\epsilon|^2}{\int_C\phi_\epsilon^2}
   \leq \Lambda(S) - \delta.
\end{equation*}
Since $\phi_\epsilon$ vanishes along $\partial C$,
it follows from the definition of $\Lambda(S)$ that $C$ is neither a disc, nor an annulus,
nor a cross cap.
Hence the Euler characteristic of $C$ is negative.

{\sc Case 2:}
Assume now that $\mathcal{R}(\phi) = \Lambda(S)$.
This is the only part of the proof which requires the regularity theory
of the nodal sets of eigenfunctions,
and it is needed to establish that the inequality in \cref{theo} is strict.

Since $\mathbb E$ is the sum of the eigenspaces of $\Delta$
with eigenvalues $\le\Lambda(S)$,
the equality $\mathcal{R}(\phi) = \Lambda(S)$ implies
that $\phi$ is an eigenfunction of $\Delta$ with eigenvalue $\Lambda(S)$.
Now it is a classical result that non-zero eigenfunctions of the Laplacian
cannot vanish of infinite order at any point; see e.g. \cite{Ar}.
Therefore, by the main result of \cite{Be},
at any critical point $z\in Z_\phi$ of $\phi$,
there are Riemannian normal coordinates $(x,y)$ about $z$,
a spherical harmonic $p=p(x,y)\ne0$ of some order $n\ge2$,
and a constant $\alpha\in(0,1)$ such that
\begin{equation*}
   \phi(x,y) = p(x,y) + O(r^{n+\alpha}),
\end{equation*}
where we write $(x,y)=(r\cos\theta,r\sin\theta)$.
By Lemma 2.4 of \cite{Che},
there is a local $C^1$-diffeomorphism $\Phi$ about $0\in\R^2$ fixing $0$ such that
\begin{equation*}
  \phi=p\circ \Phi.
\end{equation*}
Note that, up to a rotation of the $(x,y)$-plane, we have
\begin{equation*}
  p=p(x,y)=cr^n\cos n\theta
\end{equation*}
for some constant $c\ne0$.
It follows that the nodal set $Z_\phi$ of $\phi$ is a finite graph
with critical points of $\phi$ as vertices (\cite[Theorem 2.5]{Che}).
It also follows that, for any sufficiently small $\epsilon>0$,
the only critical points of $\phi$ in $\{|\phi|\le\epsilon\}$ are already contained in $Z_\phi$.
In particular, the gradient flow of $\phi$ can be used to obtain a deformation retraction
of $S\setminus Z_\phi$ onto $\{|\phi|\ge\epsilon\}$.

For any component $C$ of $S\setminus Z_\phi$,
the restriction of $\phi$ to $C$ vanishes nowhere on $C$,
and hence $\phi$ is the eigenfunction for the first Dirichlet eigenvalue of $C$.
It follows that $\lambda_0(C)=\Lambda(S)$.

Since $\phi$ is perpendicular to the constant functions,
the interior of the complement of a component $C$ as above is non-empty.
Hence $C$ can be strictly enlarged within $S$,
keeping the topological type of $C$, while strictly decreasing $\lambda_0(C)$;
see \eqref{dineq}.
It follows that no component $C$ of $S\setminus Z_\phi$ is diffeomorphic
to a disc or an annulus or a cross cap (with piecewise smooth boundary).
Thus each component of $S\setminus Z_\phi$ has negative Euler characteristic.

It follows also that the graph $Z_\phi$ does not contain non-trivial loops
which are homotopic to zero in $S$ since otherwise $S\setminus Z_\phi$
would contain a component which is a disc.
Hence, for all sufficiently small regular $\epsilon>0$,
no component of $Z_\epsilon(\phi)$ is contained in a disc
and each component of $Y_\phi(\epsilon)$ has negative Euler characteristic.
Thus $Y_\phi'(\epsilon)=Y_\phi(\epsilon)$, for all sufficiently small $\epsilon>0$.
\end{proof}

\begin{lem}\label{lemeul2}
For all regular $\epsilon>0$, we have $\chi(S)\le\chi(Y_\phi'(\epsilon))$.
\end{lem}

\begin{proof}
By definition, $Y_\phi'(\epsilon)$ and $S\setminus\mathring Y_\phi'(\epsilon)$
are smooth subsurfaces of $S$ which intersect along their common boundary,
a finite number of circles.
Hence
\begin{equation*}
    \chi(S) = \chi(Y_\phi'(\epsilon))
    + \chi(S\setminus\mathring Y_\phi'(\epsilon)),
\end{equation*}
by the Mayer-Vietoris sequence.
No component of $S\setminus\mathring Y_\phi'(\epsilon)$ is a disc
since otherwise the boundary of the disc would be a loop in $Y_\phi'(\epsilon)$
which is not homotopic to zero in $Y_\phi'(\epsilon)$,
but homotopic to zero in $S$.
This would be in contradiction to \cref{lemyef}.\ref{inco}.
It follows that $\chi(S\setminus\mathring Y_\phi'(\epsilon))\le0$.
\end{proof}

For later purposes, we want to attach signs to the components $C$ of $Y_\phi'(\epsilon)$:
We say that $C$ is \emph{positive} or \emph{negative} if $C$ is the union of
maximal essential discs with a component of $\{\phi\ge\epsilon\}$
or a component of $\{\phi\le-\epsilon\}$, respectively.
We denote by ${Y_\phi'}^+(\epsilon)$ and ${Y_\phi'}^-(\epsilon)$
the subsets of positive and negative components of $Y_\phi'(\epsilon)$, respectively.

\begin{lem}\label{lemeul4}
Let $\epsilon_1>\epsilon_2>0$ be regular.
Then
\begin{equation*}
  Y_\phi'(\epsilon_1)\subseteq Y_\phi'(\epsilon_2)
  \quad\text{and}\quad
  \chi(Y_\phi'(\epsilon_2))\le\chi(Y_\phi'(\epsilon_1)).
\end{equation*}
Moreover,
if $\chi(Y_\phi'(\epsilon_2))=\chi(Y_\phi'(\epsilon_1))$,
then $Y_\phi'(\epsilon_2)$ arises from $Y_\phi'(\epsilon_1)$
by attaching annuli and cross caps
along boundary curves of $Y_\phi'(\epsilon_1)$.
The analogous statements hold for ${Y_\phi'}^\pm(\epsilon_1)$
and ${Y_\phi'}^\pm(\epsilon_2)$ in place of $Y_\phi'(\epsilon_1)$ and $Y_\phi'(\epsilon_1)$, respectively.
\end{lem}

\begin{proof}
By definition, $Z_\phi(\epsilon_2)\subseteq Z_\phi(\epsilon_1)$.
If a component of $Z_\phi(\epsilon_1)$
is contained in the interior of an embedded closed disc,
then also all the components of $Z_\phi(\epsilon_2)$ it contains.
It follows that $Z_\phi'(\epsilon_2)\subseteq Z_\phi'(\epsilon_1)$
and hence that $Y_\phi(\epsilon_1)\subseteq Y_\phi(\epsilon_2)$.

Let $C_1$ be a component of $Y_\phi'(\epsilon_1)$
and $C$ be the component of $Y_\phi(\epsilon_2)$ which contains it.
Let $B$ be the union of the components of $Y_\phi'(\epsilon_1)$
which are contained in $C$.
Since $\epsilon_1\ne\epsilon_2$, the boundaries of $B$ and $C$ are disjoint,
by \cref{lemyef}.\ref{sign}.
Since the Euler characteristics of the components of $B$ are negative,
boundary curves of $B$ are not homotopic to zero in $B$.

Assume that $\chi(C)>\chi(B)$.
Then one of the components of $C\setminus\mathring B$ is a (closed) disc.
Then a boundary curve of $B$ would be homotopic to zero in $S$
in contradiction to the incompressibilty of $B$; see \cref{lemyef}.\ref{inco}.
We conclude that $\chi(C)\le\chi(B)$.
Since $\chi(B)<0$, we also conclude that $C\subseteq Y_\phi'(\epsilon_2)$.
Therefore $Y_\phi'(\epsilon_1)\subseteq Y_\phi'(\epsilon_2)$
and $\chi(Y_\phi'(\epsilon_2))\le\chi(Y_\phi'(\epsilon_1))$.
Equality implies that the differences $C\setminus\mathring B$ as above
consists of annuli and cross caps.

By what we just said, the last assertion follows
if ${Y_\phi'}^\pm(\epsilon_1)\subseteq{Y_\phi'}^\pm(\epsilon_2)$.
To show this, let $C_1$ be a positive component of $Y_\phi'(\epsilon_1)$
and $C$ be the component of $Y_\phi'(\epsilon_2)$ containing it.
Assume first that $C_1\ne S$, that is, that $\partial C_1\ne\emptyset$.
Now $C$ is the union of a number of maximal essential discs (with respect to $\epsilon_2$)
with a component $C_0$ of $\{\phi\ge\epsilon_2\}$ or $\{\phi\le-\epsilon_2\}$.
Since $C_1$ is incompressible in $S$ and the boundary curves of $C_1$
are not homotopic to zero in $C_1$,
$\partial C_1$ is not contained in any of the maximal discs.
Therefore $\partial C_1$ intersects $C_0$ non-trivially.
Since $\phi|_{\partial C_1}=\epsilon_1$,
we conclude that $C_0$ is a component of $\{\phi\ge\epsilon_2\}$.
Hence $C$ is positive
and therefore $Y_\phi^+(\epsilon_1)\subseteq Y_\phi^+(\epsilon_2)$.

The case $C_1=S$ follows from the Schoenfliess theorem.
The proof of the inclusion $Y_\phi^-(\epsilon_1)\subseteq Y_\phi^-(\epsilon_2)$
is similar.
\end{proof}

We want fo modify the subsurfaces ${Y_\phi'}^\pm(\epsilon)$
so that their isotopy type in $S$ becomes independent of $\epsilon$ as $\epsilon\to0$:
For any regular $\epsilon>0$,
we let $X_\phi^+(\epsilon)$ be the union of ${Y_\phi'}^+(\epsilon)$
with the components of the complement of the interior of ${Y_\phi'}^+(\epsilon)$ in $S$
which are annuli and cross caps.
Note that $\phi=\epsilon$ on the boundary of such annuli and cross caps.
We define $X_\phi^-(\epsilon)$ accordingly
and set $X_\phi(\epsilon)=X_\phi^+(\epsilon)\cup X_\phi^-(\epsilon)$.
Note that
\begin{equation}\label{yx}
  \chi(X_\phi^\pm(\epsilon)) = \chi({Y_\phi'}^\pm(\epsilon))
  \quad\text{and}\quad
  \chi(X_\phi(\epsilon)) = \chi({Y_\phi'}(\epsilon)).
\end{equation}
By construction and \cref{lemyef}.\ref{sign},
$\phi|_{\partial C}=\pm\epsilon$ for any component $C$ of $X_\phi^\pm(\epsilon)$.
Observe that $X_{-\phi}^+(\epsilon)=X_\phi^-(\epsilon)$,
and accordingly for ${Y_\phi'}^{\pm}(\epsilon)$.

\begin{lem}\label{lemeul6}
Let $\epsilon_1>\epsilon_2>0$ be regular
and suppose that $\chi(X_\phi(\epsilon_1))=\chi(X_\phi(\epsilon_2))$.
Then $(S,X^+_\phi(\epsilon_1),X^-_\phi(\epsilon_1))$
is isotopic to $(S,X^+_\phi(\epsilon_2),X^-_\phi(\epsilon_2))$;
that is, there is a diffeomorphism of $S$ which is isotopic to the identity and which
restricts to a diffeomorphism between $X_\phi^+(\epsilon_1)$ and $X_\phi^+(\epsilon_2)$
respectively between $X_\phi^-(\epsilon_1)$ and $X_\phi^-(\epsilon_2)$.
\end{lem}

\begin{proof}
By \eqref{yx}, we have $\chi(Y_\phi'(\epsilon_1))=\chi(Y_\phi'(\epsilon_2))$.
Hence ${Y_\phi'}^\pm(\epsilon_2)$ arises from ${Y_\phi'}^\pm(\epsilon_1)$
by attaching annuli and cross caps, by \cref{lemeul4}.
The point of the argument below is that all boundary curves of ${Y_\phi'}^\pm(\epsilon_2)$
arise by attaching an annulus $A$ to a boundary curve of ${Y_\phi'}^\pm(\epsilon_1)$.
Then $\phi$ is equal to $\epsilon_1$ on one of the boundary curves of $A$
and equal to $\epsilon_2$ on the other.

Without loss of generality, we only consider the $X^+$-spaces.
It suffices to show that $X_\phi^+(\epsilon_2)$ arises from $X_\phi^+(\epsilon_1)$
by attaching annuli $A$ such that $\phi$ is equal to $\epsilon_1$
on one of the boundary curves of $A$ and equal to $\epsilon_2$ on the other.

There are several cases in the passage from the $Y'$-spaces to the $X$-spaces.
We exemplify the argument in one of the cases.

Suppose that, in the passage from ${Y_\phi'}^+(\epsilon_1)$ to $X_\phi^+(\epsilon_1)$,
an annulus $A$ is attached to ${Y_\phi'}^+(\epsilon_1)$ such that $\phi$ is
equal to $\epsilon_1$ on the boundary curves of $A$.
Then $\phi=\epsilon_1$ on $\partial A$ and either $\phi>\epsilon_2$ on $A$
or else, by \cref{lemeul4}, there are disjoint annuli $A',A''\subseteq A$,
each of them sharing a boundary curve with $A$,
such that $\phi$ is equal to $\epsilon_2$ on the other boundary curve
and such that $A'$ and $A''$ belong to ${Y_\phi'}^+(\epsilon_2)$.
By \cref{annulus}, we get an annulus $A'''$ in $A$ between $A'$ and $A''$
and sharing one of its boundary curves with $A'$ and the other with $A''$.
In particular, $\phi$ is equal to $\epsilon_2$ on both boundary curves of $A'''$.
We conlude that $A=A'\cup A'''\cup A''$ belongs to $X_\phi(\epsilon_2)$.
\end{proof}

\newpage
\section{End of proof of \cref{theo}}
By the above Lemmas \ref{lemeul} -- \ref{lemeul6},
we obtain a partition of the unit sphere $\mathbb{S}$ in $\mathbb E$ into the subsets
\begin{equation}
  \mathcal{C}_i := \{ \phi \in \mathbb{S} \mid
  \text{$\chi(X_\phi(\epsilon)) = i$ for all sufficiently small $\epsilon>0$}\},
\end{equation}
where $\chi(S)\le i<0$.
By definition, $\phi\in\mathcal{C}_i$ if and only if $-\phi\in\mathcal{C}_i$.
Hence $\mathcal{C}_i$ is the preimage of the subset $\mathcal{B}_i=\pi(\mathcal C_i)$
in the projective space $\mathbb{P}=\mathbb{S}/\pm\id$
under the canonical projection $\pi:\mathbb{S} \rightarrow \mathbb{P}$.

\begin{lem}\label{lemiso3}
Let $\epsilon>0$ and $U\subseteq\mathbb S$ be the subset of $\phi$
such that $\epsilon$ is $\phi$-regular.
Then $U$ is open and the isotopy types of
$(S,X^+_\phi(\epsilon),X^-_\phi(\epsilon))$
are locally constant as functions of $\phi\in U$.
\end{lem}

\begin{proof}
Note that $U$ is open since $S$ is compact.
Consider the map
\begin{equation*}
  F\colon U\times S\to\R, \quad F(\phi,z)=\phi(z).
\end{equation*}
Since $\mathbb E$ is finite dimensional, any two norms on $\mathbb E$ are equivalent.
In particular, $F$ is continuously differentiable.
If $\phi\in U$ and $z\in S$ satisfy $\phi(z)=\epsilon$,
then $d\phi_z\ne0$, and hence $dF_{(\phi,z)}\ne0$.
Choose a vector $v\in T_zS$ with $dF_z(v)\ne0$
and coordinates $(x,y)$ of $S$ about $z$ such that $z=(0,0)$ and $v=\partial/\partial y$.
Then, by the implicit function theorem,
there are open intervals $I\ni0$ and $J\ni\epsilon$,
an open neighborhood $W$ of $\phi$ in $\mathbb E$,
and a smooth function $y\colon W\times I\times J\to\R$
such that $F(\psi,x,y(\psi,x,\tau))=\tau$ for all $(\psi,x,\tau)\in W\times I\times J$.
It follows that the families of curves $\{\phi=\epsilon\}$ depend smoothly on $\phi\in U$;
and similarly for $-\epsilon$.
The claim of \cref{lemiso3} now follows easily from the construction of the $X^\pm$-spaces.
\end{proof}

Now we are ready for the final steps of the proof of \cref{theo}

\begin{proof}[Proof of \cref{theo}]
For $\mathcal B_i=\pi(\mathcal C_i)$ as above,
we show that $\pi\colon\mathbb S\to\mathbb P$ is trivial over $\mathcal B_i$.
To that end, we note that, for $\phi\in\mathcal C_i$, we have $-\phi\in\mathcal C_i$
and that
\begin{equation*}
  (S,X^+_{-\phi}(\epsilon),X^-_{-\phi}(\epsilon))
  = (S,X^-_\phi(\epsilon),X^+_\phi(\epsilon)).
\end{equation*}
Now $(S,X^+\phi(\epsilon),X^-_\phi(\epsilon))$
is not isotopic to $(S,X^-_\phi(\epsilon),X^+_\phi(\epsilon))$, by \cref{ivan}.
Hence the partition of $\mathcal C_i$ into the open subsets $U_{i,j}$
with the same isotopy type has the property that $U_{i,j}\cap -U_{i,j}=\emptyset$.
Note also that $U_{i,j}\cup -U_{i,j}$ is the preimage of a subset $V_{i,j}$ in $\mathbb P$.
It follows that $\pi|_{\mathcal B_i}$ is trivial.
We conclude that $-\chi(S)>\dim\mathbb P=\dim\mathbb E-1$, by Lemma 8 in \cite{Se}
(see also the final paragraph in the proof of Lemme 5 in \cite{OR}).
\end{proof}

\newpage
\appendix
\section{On the topology of surfaces}

For the convenience of the reader (and the authors),
we collect some facts from the topology of surfaces.

In what follows, let $S$ be a surface of finite type, that is,
$S$ is diffeomorphic to the interior of a compact surface with boundary,
with Euler number $\chi(S)\le0$.
In other words, $S$ is of finite type,
but is not diffeomorphic to the sphere or the real projective plane.
In the orientable case, the first two assertions
are Corollary A.7 and Proposition A.11 in \cite{Bu2}.

\begin{lem}\label{lemdisc}
Any homotopically trivial simple closed curve in $S$ bounds an embedded disc. \qed
\end{lem}

\begin{lem}\label{annulus}
Let $c_0$ and $c_1$ be smooth two-sided simple closed curves in $S$
which are freely homotopic (up to their orientation) and which do not intersect.
Then $c_0\cup c_1$ bounds an embedded annulus in $S$. \qed
\end{lem}

\begin{lem}\label{lemdisc2}
Let $C\subseteq S$ be a connected subsurface with smooth boundary
which is closed as a subset of $S$.
Assume that $C$ contains a closed curve which is homotopic to zero in $S$, but not in $C$.
Then $S\setminus\mathring{C}$ contains a connected component
which is diffeomorphic to a closed disc.
\end{lem}

\begin{proof}
Without loss of generality we may assume that $\mathring{C}$
contains a smoothly immersed simple closed curve $c$
which is homotopic to zero in $S$, but not in $C$.
Now $c$ bounds a smooth disc $D$ in $S$, by \cref{lemdisc},
and $B=C\cap D$ is a smooth and closed subsurface in $D$.
Furthermore, $B$ is connected since $c=\partial D\subseteq B$
and $B\ne D$ since otherwise $c$ would be homotopic to zero in $C$.
Hence the interior of $D$ contains boundary circles $c'$ of $B$,
and the interior $D'$ of any such $c'$ in $D$ is disjoint from $C$.
By the Schoenfliess theorem, any such $D'$ is diffeomorphic to a disc.
\end{proof}

A subsurface $C\subseteq S$ is called \emph{incompressible} in $S$
if any closed curve in $C$, which is homotopic to zero in $S$,
is already homotopic to zero in $C$.

\begin{cor}\label{corinco}
Let $C\subseteq S$ be a connected subsurface with boundary
which is closed as a subset of $S$.
Assume that no component of $S\setminus\mathring{C}$ is diffeomorphic to a closed disc.
Then $C$ is incompressible in $S$.
\end{cor}

For  a proof of the following result, we refer to Chapter 1 of \cite{Iv}.

\begin{thm}\label{ivan}
Let $S$ be a compact and connected surface with $\chi(S)<0$
and $L\subseteq S$ be a closed one-dimensional submanifold.
Let $F\colon S\to S$ be a diffeomorphism which is isotopic to the identity
and such that $F(L)=L$.
Then $F$ leaves all components of $L$ and $S\setminus L$ invariant.
\end{thm}



\end{document}